\documentclass[12pt]{amsart}
\usepackage{tikz,array,verbatim}
\usepackage{amsfonts, amsmath, latexsym, epsfig, caption}
\usepackage{amssymb, color}

\usepackage{epsf}
\usepackage{array}
\usepackage{ragged2e}
\usepackage{hyperref}
\usepackage[margin=1in]{geometry}
\usepackage{longtable}

\newcommand{\calP}{\ensuremath{\mathcal{P}}}
\newcommand{\G}{\Gamma}

\newtheorem{theorem}{Theorem}[section]

\newtheorem{lemma}[theorem]{Lemma}
\newtheorem{remark}[theorem]{Remark}

\theoremstyle{definition}

\begin{document}

\title[Symmetric polytopes whose automorphism
groups are 2-groups]{Symmetric polytopes whose automorphism \\groups are 2-groups}

\author[G. Cunningham]{Gabe Cunningham}
\address{Gabe Cunningham, Applied Mathematics, School of Computing and Data Science, Wentworth Institute of Technology, Boston, MA 02115, USA}
\email{cunninghamg1@wit.edu}

\author[Y.-Q. Feng]{Yan-Quan Feng}
\address{Yan-Quan Feng, Department of Mathematics, Beijing Jiaotong University, Beijing, 100044, P.R. China}
\email{yqfeng@bjtu.edu.cn}

\author[D.-D. Hou]{Dong-Dong Hou}
\address{Dong-Dong Hou, Shanxi Key Laboratory of Cryptography and Date Security, Department of Mathematics, Shanxi Normal University, Taiyuan, 100044, P.R. China}
\email{holderhandsome@bjtu.edu.cn}

\author[E. Schulte]{\\ Egon Schulte}
\address{Egon Schulte, Department of Mathematics, Northeastern University, Boston, MA, 02115, USA}
\email{e.schulte@northeastern.edu}

\date{ \today } 
\maketitle 

\begin{abstract}
The present work investigates regular, semiregular, and chiral polytopes of any rank $d\geq 3$, whose automorphism groups are 2-groups.  There is a large variety of rather small finite regular or  alternating semiregular polytopes with automorphism groups of 2-power order:\ for such polytopes with toroidal sections of rank 3, the various sections of rank $3$ can be entirely prescribed (possibly with one exception in the semiregular case). It is also shown that having a 2-group as automorphism group is hereditary under taking universal extensions:\ the universal extension of a given regular, chiral, or alternating semiregular polytope with a finite or infinite 2-group as automorphism group, is a polytope of one rank higher with an infinite 2-group an automorphism group. 
\bigskip\medskip

\noindent
Key Words:  Regular polytope, Automorphism group, 2-Group, Semiregular, Chiral
\medskip

\noindent
MSC Subject Classification (2020): 20B25, 52B15, 51M20
\end{abstract}

\section{Introduction}
\label{intro}

Abstract polytopes have been the subject of considerable interest for the past few decades (McMullen \& Schulte~\cite{McMSch2002}, McMullen~\cite{McM2020}, Pellicer~\cite{Pel2025}). They are combinatorial structures with distinctive geometric, algebraic, or topological characteristics. A large body of literature focuses on the classification of regular, chiral, or other highly symmetric abstract polytopes, usually with a focus on the polytopes themselves. A separate line of inquiry begins with the groups, and attempts to enumerate the regular or chiral polytopes whose automorphism group is of a preassigned type; for example, a symmetric group~\cite{CFL2024}, an alternating group~\cite{CFLM2017,CHO2024}, or another almost simple group such as a projective linear group~\cite{BroLee2016,BroVic2010,LeeSch2007}.

The present work investigates regular, semiregular, and chiral polytopes whose automorphism group is a 2-group. This continues a direction of research, initiated by a problem posed in~\cite{SchWei2006} but begun in earnest by Hou, Feng, Leemans and Qu~\cite{HFL2019a,HFL2019b,HFL2020,HFL2025,HFLQ2025}, as well as Gomi, Loyola and De Las Penas~\cite{GLD2018,Loy2018} for special cases. Our interest is primarily in polytopes with finite 2-groups, but infinite 2-groups will occur as well.

Our results show that there is a large variety of rather small finite regular and semiregular polytopes with a group of 2-power order. Specifically, for $d$-polytopes with toroidal sections of rank 3 ($d\geq 4$), the various sections of rank $3$ can be entirely prescribed, except possibly the $(d-4)$-co-faces in the semiregular case:\ any combination of toroidal regular polyhedra, each with a 2-group, occurs as the collection of sections of rank 3 in a small regular or alternating semiregular polytope whose automorphism group is a 2-group.

We also establish that the free (universal) extensions of regular or chiral $d$-polytopes whose automorphism group is a finite or infinite 2-group, is a regular or chiral $(d+1)$-polytope, respectively, whose automorphism group is an infinite 2-group. A similar result also holds for the universal alternating semiregular $(d+1)$-polytope with facets given by a compatible pair of regular $d$-polytopes $\mathcal P$ and $\mathcal Q$ with finite or infinite 2-groups.

The paper is organized as follows. In Sections~\ref{bano} and~\ref{FAP}, we give the necessary background and review the flat amalgamation property for regular polytopes. Following a brief discussion of toroidal regular polyhedra with 2-groups in Section~\ref{44type}, we then exploit the flat amalgamation property to construct small regular and semiregular polytopes with prescribed toroidal sections of rank 3 and a 2-power automorphism group; this is done in Sections~\ref{regpols2} and~\ref{semiregpols2}, respectively. Section~\ref{inf2group} describes our results about polytopes with infinite 2-groups. The final Section~\ref{powerpols} briefly comments on regular power-polytopes with groups of 2-power order.

\section{Basic Notions}
\label{bano}

In this section we briefly review basic concepts and results about abstract regular polytopes, mostly following \cite[Chs. 2,\,3]{McMSch2002}. 

An (\emph{abstract\/}) \textit{polytope of rank\/} $d$ is a partially ordered set $\mathcal{P}$ with a strictly monotone rank function that takes values in $\{-1,0, \ldots, d\}$. An element of rank~$j$ is called a \textit{$j$-face\/} of $\mathcal{P}$, and a face of rank $0$, $1$ or $d-1$ is also called a \textit{vertex\/}, \textit{edge\/} or \textit{facet\/}, respectively. The polytope $\mathcal{P}$ has a unique least face $F_{-1}$ (of rank $-1$) and a unique greatest face $F_d$ (of rank $d$), called \textit{improper\/} faces of $\mathcal{P}$. A \textit{flag\/} is a maximal totally ordered subset of $\mathcal{P}$. Each flag contains exactly $d+2$ faces, one for each rank, including $F_{-1}$ and $F_{d}$. Two flags are said to be \textit{adjacent} if they differ in one face; if this face is a $j$-face, we also call the flags $j$-adjacent. A polytope $\mathcal{P}$ is \textit{strongly flag-connected}, in the sense that any two flags $\Phi$ and $\Psi$ of $\mathcal{P}$ can be joined by a finite sequence of successively adjacent flags, each containing $\Phi \cap \Psi$. Finally, $\mathcal{P}$ satisfies the \textit{diamond condition\/}:\ whenever $F \leq G$, with $F$ a $(j-1)$-face and $G$ a $(j+1)$-face for some~$j$, there are exactly two $j$-faces $H$ with $F \leq H \leq G$. The diamond condition, rephrased, is saying that every flag $\Phi$ has a unique $j$-adjacent flag for each $j=0, \dots, d-1$, denoted $\Phi^j$.  An abstract polytope of rank 3 is also called an (\emph{abstract}) \emph{polyhedron}.

For any two faces $F$ of rank $j$ and $G$ of rank $k$ with $F \leq G$, we call
\[G/F := \{ H \in \mathcal{P}\, | \, F \leq H \leq G \}\] 
a \textit{section} of $\mathcal{P}$. This is a $(k-j-1)$-polytope in its own right. We often identify a face $F$ with the section $F/F_{-1}$. For a $j$-face $F$ the section $F_{d}/F$ is called the \textit{co-$j$-face\/} (or just \textit{co-face\/}) of $\mathcal{P}$ at $F$, or the \emph{vertex-figure\/} or \emph{edge-figure\/} of $\mathcal{P}$ at $F$ if $F$ is a vertex or an edge, respectively. A co-$j$-face has rank $d-j-1$.

A polytope $\mathcal{P}$ is said to be \emph{regular\/} if its automorphism group $\Gamma(\mathcal{P})$ is transitive on the flags. In this case $\Gamma(\mathcal{P})$ acts regularly on the flags, as the flag stabilizers in the automorphism group of a polytope are necessarily trivial. If $\mathcal{P}$ is a regular polytope, then $\Gamma(\mathcal{P})$ acts transitively on the faces of each rank, all sections of $\mathcal{P}$ are also regular, and comparable sections (defined by pairs of faces of matching ranks) are isomorphic. A regular polytope is said to be of (Schl\"afli) type $\{p_1,\ldots,p_{d-1}\}$ if for each $i=1,\ldots,d-1$ and each incident pair $F,G$ of $(i-2)$-face and $(i+1)$-face, the section $G/F$ is isomorphic to a $p_i$-gon (possibly $p_{i}=\infty$). 

The group $\Gamma(\mathcal{P})$ of a regular polytope $\mathcal{P}$ has a well-behaved system of \textit{distinguished generators\/}. In fact, $\Gamma(\mathcal{P})$ is generated by involutions $\rho_0,\ldots,\rho_{d-1}$, where $\rho_i$ maps a (fixed) \textit{base\/} flag $\Phi$ to its $i$-adjacent flag $\Phi^i$. These generators satisfy (at least) the standard Coxeter-type relations for Coxeter groups with string diagrams, 
\begin{equation}
\label{standardrel}
(\rho_i \rho_j)^{p_{ij}} = 1 \;\; \textrm{ for } i,j=0, \ldots,d-1,
\end{equation}
where $p_{ii}=1$, $p_{ji} = p_{ij} =: p_{i+1}$ if $j=i+1$, and $p_{ij}=2$ otherwise. The numbers $p_j$ are the entries of the Schl\"afli symbol $\{p_{1},\ldots,p_{d-1}\}$. Moreover, the following {\em intersection property\/} holds,
\begin{equation}
\label{intprop}
\langle \rho_i \mid i \in K \rangle \cap \langle \rho_i \mid i \in J \rangle
= \langle \rho_i \mid i \in {K \cap J} \rangle
  \;\; \textrm{ for } K,J \subseteq \{0,1,\ldots,d-1\}.
\end{equation}
Further, the elements $\sigma_{1},\ldots,\sigma_{d-1}$, with $\sigma_{i}:=\rho_{i-1}\rho_{i}$ for $i=1,\ldots,d-1$, generate the {\em rotation subgroup\/} $\Gamma^{+}(\mathcal{P})$ of $\Gamma(\mathcal{P})$, which is of index at most~$2$. This subgroup consists of all elements of $\Gamma(\mathcal{P})$ that can be written as a product of an even number of generators $\rho_i$. A regular polytope $\mathcal{P}$ is {\em directly regular\/} if the index of $\Gamma^{+}(\mathcal{P})$ in $\Gamma(\mathcal{P})$ is $2$.

A \textit{string C-group\/} is a group $\Gamma = \langle \rho_0,\ldots, \rho_{d-1} \rangle$ whose generators satisfy (\ref{standardrel}) and the intersection property (\ref{intprop}); here, the ``C'' stands for ``Coxeter'', though not every string C-group is a Coxeter group. The string C-groups are precisely the automorphism groups of regular polytopes (see \cite[Ch. 2E]{McMSch2002}). We usually identify a regular polytope with its automorphism (string C-) group.

We write $[p_{1},p_{2},\ldots,p_{d-1}]$ for the Coxeter group whose underlying Coxeter diagram is a string with $d$ nodes and $d-1$ branches labeled $p_{1},p_{2},\ldots,p_{d-1}$. (Here we regard the $i$-th branch of the string as missing if $p_{i}=2$.) This is a string C-group and is the automorphism group of the \textit{universal} regular $n$-polytope $\{p_{1},\ldots,p_{d-1}\}$ (see \cite[Ch. 3D]{McMSch2002}). 

Assembling regular polytopes from polytopes of lower rank is usually a difficult problem. If a regular $d$-polytope $\mathcal{P}$ has facets isomorphic to $\mathcal{P}_1$ and vertex-figures isomorphic to $\mathcal{P}_2$, then $\mathcal{P}_1$ and $\mathcal{P}_2$ must necessarily be $(d-1)$-polytopes that are \textit{compatible}, in the sense that the vertex-figures of $\mathcal{P}_1$ must be isomorphic to the facets of $\mathcal{P}_2$. The compatibility of a pair of regular $(d-1)$-polytopes $\mathcal{P}_1$ and $\mathcal{P}_2$ is only necessary but not sufficient to guarantee the existence of a regular $d$-polytope with facets $\mathcal{P}_1$ and vertex-figures $\mathcal{P}_2$. However, if such a polytope indeed exists, then there is also a universal such polytope, denoted $\{\mathcal{P}_1,\mathcal{P}_2\}$, which covers all other regular polytopes with facets $\mathcal{P}_1$ and vertex-figures $\mathcal{P}_2$ (see \cite[Thm.\,4A2]{McMSch2002}). Generalizing the concept of the Schl\"afli symbol, we slightly abuse notation and say that a regular polytope $\mathcal{P}$ is of \textit{type\/} $\{\mathcal{P}_1,\mathcal{P}_2\}$ if $\mathcal{P}$ has facets $\mathcal{P}_1$ and vertex-figures $\mathcal{P}_2$.

We will also use a \textit{generalized Schl\"afli symbol\/} recording the structure of the sections of rank 3. Recall that if $\Phi=\{F_{-1},F_0,\ldots,F_d\}$ is a flag of a regular $d$-polytope $\mathcal P$, the sequence of sections of rank 3 determined by $\Phi$, 
\begin{equation}
\label{rk3sections}
F_3/F_{-1},\ F_4/F_0,\ F_5/F_1,\ \ldots,\ F_{d}/F_{d-4},
\end{equation}
captures the local geometry of the sections of rank 3 of $\mathcal P$, in the sense that each section $G/F$ given by a $j$-face $G$ and a $(j-4)$-face $F$ with $F\leq G$ is isomorphic to $F_j/F_{j-4}$. We then say that $\mathcal{P}$ is of type $\{\mathcal{P}_3,\ldots,\mathcal{P}_{d}\}$ if $\mathcal{P}_j$ is a regular 3-polytope isomorphic to the section $F_{j}/F_{j-4}$ of $\mathcal{P}$ for each $j=3,\ldots,d$. If $d=4$, this is consistent with our previous notation (apart from renaming the subscripts). Note that our notation is not meant to imply universality of $\mathcal{P}$ among regular polytopes with the same generalized Schl\"afli symbol.

A $d$-polytope $\mathcal{P}$ is called \emph{chiral\/} if its automorphism group $\Gamma(\mathcal{P})$ has two flag orbits such that any two adjacent flags are in distinct orbits (see \cite{Pel2025,SchWei1991}). If $\mathcal{P}$ is a chiral $d$-polytope, then $\Gamma(\mathcal{P})$ is generated by elements $\sigma_1,\ldots,\sigma_{d-1}$ where $\sigma_i$ maps a (fixed) base flag $\Phi$ of $\mathcal{P}$ to $(\Phi^{i})^{i-1}$, the $(i-1)$ adjacent flag of the $i$-adjacent flag $\Phi^i$ of $\Phi$. These generators satisfy (at least) the relations
\begin{equation}
\label{chiralrel}
\sigma_i^{p_i} = (\sigma_i\sigma_{i+1}\cdot\ldots\cdot\sigma_j)^{2} = \epsilon
\;\; \textrm{ for } i,j=1,\dots,d-1,  \textrm{ with } i<j,
\end{equation}
where as before the numbers $p_i$ determine the Schl\"afli symbol $\{p_1,\ldots,p_{d-1}\}$ of~$\mathcal{P}$. The generators of a chiral polytope also satisfy an intersection property, which is more complicated than (\ref{intprop}).

A $d$-polytope $\mathcal S$ is called {\it semiregular\/} if it has regular facets and its automorphism group $\Gamma(\mathcal{S})$ is transitive on vertices. A semiregular polytope $\mathcal S$ is {\it alternating} if its facets (all regular) are of two kinds $\mathcal P$ and $\mathcal Q$, which further occur in alternating fashion around any $(d-3)$-face in $\mathcal S$ (see \cite{MonSch2012,MonSch2020}). 
We allow the case that $\mathcal{P}$ and $\mathcal Q$ are isomorphic, but encounter this possibility  only in situations where $\Gamma(\mathcal{S})$ has a subgroup of index 2 with only two facet orbits such that the facet orbits around any $(d-3)$-face in $\mathcal S$ occur in alternating fashion. Simple examples of alternating semiregular polyhedra are the cuboctahedron and the icosidodecahedron, with two triangles and two squares respectively pentagons occurring in alternating fashion around any vertex. (Coxeter~\cite{Cox1973} calls such polyhedra ``quasiregular".) In rank 3, all facets are 2-polytopes and thus regular, so an abstract polyhedron is semiregular if and only $\Gamma(\mathcal{S})$ is transitive on vertices. 

A $d$-polytope $\mathcal{P}$ is called {\em flat\/} if each of its vertices is incident with each of its facets.  More generally, if $0 \leq k < l \leq d-1$, then $\mathcal{P}$ is said to be {\em $(k,l)$-flat\/} if each of its $k$-faces is incident with each of its $l$-faces. A flat polytope is $(0,d-1)$-flat.
\smallskip

In this paper, we study polytopes whose automorphism groups are 2-groups. We are mostly interested in the case of finite 2-groups (in which case the order is a power of 2), but infinite 2-groups will occur as well. 

\section{The FAP and flat amalgamation}
\label{FAP}

In this section, we briefly review the flat amalgamation property, or FAP for short, as well as the amalgamation process for regular polytopes described in \cite[Sects.\,4E,\,4F]{McMSch2002}.

Let $\Gamma = \langle\rho_{0},\ldots,\rho_{d-1}\rangle$ be any group generated by involutions such that $(\rho_{i}\rho_{j})^{2} = 1$ for all~$i,j$ with $|i-j| \geq 2$. We usually take $\Gamma$ to be the group of a regular polytope.  Note that if $\varphi\in \Gamma$ is such that 
$\varphi = \varphi_{0}\rho_{j_{1}}\varphi_{1}\rho_{j_{2}} \cdots 
\varphi_{m-1}\rho_{j_{m}}\varphi_{m}$, 
then we can write 
\[\varphi  = \psi_{0}(\psi_{1}^{-1}\rho_{j_{1}}\psi_{1}) \cdots
(\psi_{m}^{-1}\rho_{j_{m}}\psi_{m}),  \]
where  $\psi_{i} := \varphi_{i}\varphi_{i+1} \cdots \varphi_{m}$ for $i = 0,\ldots,m$. 
Here, if $\varphi_{0},\ldots,\varphi_{m}$ belong to some specified subgroup
of $\Gamma$, then $\psi_{0},\psi_{1},\ldots,\psi_{m}$ also belong to that subgroup.

For $k = 0,\ldots,d-1$ define
\[\begin{array}{rclcl}
N_{k}^{+} & := & N_{k,\ldots,d-1} & := & \langle\varphi^{-1}\rho_{i}\varphi \mid i \geq k,\, \varphi \in
\Gamma\rangle,  \\[.03in]
N_{k}^{-} & := & N_{0,\ldots,k} & := & \langle\varphi^{-1}\rho_{i}\varphi \mid i  \leq k,\,\varphi \in
\Gamma\rangle.
\end{array}\]
These subgroups are the normal closures of $\{\rho_{k},\ldots,\rho_{d-1}\}$ and
$\{\rho_{0},\ldots,\rho_{k}\}$, respectively, in $\Gamma =
\langle\rho_{0},\ldots,\rho_{d-1}\rangle$, and satisfy 
\begin{equation}
\label{nkplus}
\Gamma = N_{k}^{+} \langle\rho_{0},\ldots,\rho_{k-1}\rangle =
\langle\rho_{0},\ldots,\rho_{k-1}\rangle N_{k}^{+},
\end{equation}
\begin{equation}
\label{nkminus}
\Gamma = N_{k}^{-} \langle\rho_{k+1},\ldots,\rho_{d-1}\rangle =
\langle\rho_{k+1},\ldots,\rho_{d-1}\rangle N_{k}^{-}.
\end{equation}

Now let $\mathcal{P}$ be a regular $d$-polytope with automorphism group $\Gamma(\mathcal{P}) =\langle\rho_{0},\ldots,\rho_{d-1}\rangle$, and let $0 \leq k \leq d-1$.   We say that $\mathcal{P}$ has the {\em FAP with respect to its $k$-faces\/} if the above product in (\ref{nkplus}) is semidirect; that is, \[\Gamma(\mathcal{P}) \cong N_{k}^{+} \rtimes \langle\rho_{0},\ldots,\rho_{k-1}\rangle,\]
or, equivalently, the subgroups $N_{k}^{+}$ and $\langle\rho_{0},\ldots,\rho_{k-1}\rangle$ intersect trivially. In this case, if $k=d-1$ we also say that $\mathcal{P}$ has the {\em FAP with respect to its facets\/}.

Dually, $\mathcal{P}$ is said to have the {\em FAP with respect to the co-$k$-faces\/} if the product in (\ref{nkminus}) is semi-direct; that is,
\[\Gamma(\mathcal{P}) \cong N_{k}^{-} \rtimes \langle\rho_{k+1},\ldots,\rho_{d-1}\rangle,\]
or, equivalently, $N_{k}^{-}$ and $\langle\rho_{k+1},\ldots,\rho_{n-1}\rangle$ have trivial intersection. In this case, if $k=0$ or $k=1$, we say that $\mathcal{P}$ has the {\em FAP with respect to its vertex-figures\/} or {\em edge-figures\/}, respectively. Then a regular $d$-polytope $\mathcal{P}$ has the FAP with respect to its co-$k$-faces if and only if its dual $\mathcal{P}^{*}$ has the FAP with respect to its $(d-k-1)$-faces.  (Recall that the dual $\mathcal{P}^{*}$ of $\mathcal{P}$ is the polytope obtained from $\mathcal{P}$ by leaving the face-set of $\mathcal{P}$ unchanged, but reversing the partial order on $\mathcal{P}$.)

We often make use of the fact that if a regular $d$-polytope $\mathcal{P}$ has the
FAP with respect to its $k$-faces, then the $j$-faces of $\mathcal{P}$ with $j>k$ also have the FAP with respect to their $k$-faces, and the co-$j$-faces of $\mathcal{P}$ with $j<k$ have the FAP with respect to their $(k-j-1)$-faces. A similar, dual statement holds for polytopes that have the FAP with respect to their co-$k$-faces.

Lemma~\ref{lemher} below will show that the FAP is also hereditary in another sense. The proof requires the following simple technical lemma.

\begin{lemma}
\label{lemABC}
Let $G$ be a group, and let $A,B,C$ be subgroups of $G$ such that $G=AB$, $A\cap B=C\cap B=1$, $B\unlhd G$, $C\unlhd A$, and $CB\unlhd AB$. Then 
\[ G/CB = AB/CB \cong A/C.\]
\end{lemma}

\begin{proof} 
By the isomorphism theorems for groups,
\[(AB)/B \cong A/(A\cap B) \cong A,\;\; (CB)/B \cong C/(C\cap B) \cong C\] 
and therefore
\[G/CB = AB/CB \cong (AB/B)/(CB/B)\cong A/C.\]
\end{proof}
\smallskip

The first part of the following lemma will be used in the proof of Theorem~\ref{flat44thm}. The second part is the corresponding dual statement.

\begin{lemma} 
\label{lemher}
Let $\mathcal{P}$ be a regular $d$-polytope, $d\geq 4$. \\[.03in]
(a) Let $0\leq k\leq d-3$ and $0\leq l\leq d-k-4$. If $\mathcal{P}$ has the FAP with respect to its co-$k$-faces, and the co-$k$-faces have the $FAP$ with respect to their co-$l$-faces, then $\mathcal{P}$ has the FAP with respect to its co-$(k+l+1)$-faces.\\[.03in]
(b) Let $2\leq l<k\leq d-1$. If $\mathcal{P}$ has the FAP with respect to its $k$-faces, and the $k$-faces have the $FAP$ with respect to their $l$-faces, then $\mathcal{P}$ has the FAP with respect its $l$-faces.
\end{lemma}

\begin{proof}
Let $\Gamma(\mathcal{P}) = \langle\rho_{0},\ldots,\rho_{d-1}\rangle$. As the two parts are related by duality, it suffices to prove the first part. 

The group of the co-$k$-face at the base $k$-face of $\mathcal{P}$ is given by $\langle\rho_{k+1},\ldots,\rho_{d-1}\rangle$. Then our assumptions on $\mathcal{P}$ and its co-$k$-faces imply that 
\[\Gamma(\mathcal{P}) \cong N_{k}^{-} \rtimes \langle\rho_{k+1},\ldots,\rho_{d-1}\rangle\]
and
\[\langle\rho_{k+1},\ldots,\rho_{d-1}\rangle =
\hat{N}_{k+1,\ldots,k+l+1} \rtimes \langle\rho_{k+l+2},\ldots,\rho_{d-1}\rangle,\]
where $\hat{N}_{k+1,\ldots,k+l+1}$ is defined as the normal closure of $\{\rho_{k+1},\ldots,\rho_{k+l+1}\}$ in the subgroup $\langle\rho_{k+1},\ldots,\rho_{d-1}\rangle$. We claim that 
\begin{equation}
\label{herNs}
N_{k+l+1}^{-} = N_{k}^{-}\hat{N}_{k+1,\ldots,k+l+1},
\end{equation}
or equivalently, $N_{k+l+1}^{-}/N_{k}^{-}\cong \hat{N}_{k+1,\ldots,k+l+1}$.
In fact, by the FAP of $\mathcal{P}$ with respect to its co-$k$-faces, $N_k^{-}$ intersects the subgroup $\hat{N}_{k+1,\ldots,k+l+1}$  of $\langle\rho_{k+1},\ldots,\rho_{k+l+1}\rangle$ 
trivially and therefore, 
\[\begin{array}{rcl}
N_{k+l+1}^{-}/ N_{k}^{-}&= &
\langle\rho_{0},\ldots,\rho_{k+l},\varphi\rho_{k+l+1}\varphi^{-1}\!\mid\! \varphi\in\langle\rho_{k+l+2},\ldots,\rho_{d-1}\rangle\rangle/N_{k}^{-}\\
&=&
\langle\rho_{0}N_{k}^{-},\ldots,\rho_{k+l}N_{k}^{-},\varphi\rho_{k+l+1}\varphi^{-1}N_{k}^{-}\mid \varphi\in\langle\rho_{k+l+2},\ldots,\rho_{d-1}\rangle\rangle\\
&=&
\langle\rho_{k+1}N_{k}^{-},\ldots,\rho_{k+l}N_{k}^{-},\varphi\rho_{k+l+1}\varphi^{-1}N_{k}^{-}\mid \varphi\in\langle\rho_{k+l+2},\ldots,\rho_{d-1}\rangle\rangle\\
&=&
\langle\rho_{k+1},\ldots,\rho_{k+l},\varphi\rho_{k+l+1}\varphi^{-1}\mid \varphi\in\langle\rho_{k+l+2},\ldots,\rho_{d-1}\rangle\rangle N_{k}^{-}/N_{k}^{-}\\
&=&\hat{N}_{k+1,\ldots,k+l+1} N_{k}^{-}/N_{k}^{-}\\
&=&\hat{N}_{k+1,\ldots,k+l+1}/(\hat{N}_{k+1,\ldots,k+l+1}\cap N_{k}^{-})\\
&\cong&\hat{N}_{k+1,\ldots,k+l+1}.
\end{array}\]
Thus, $N_{k+l+1}^{-} = N_{k}^{-}\hat{N}_{k+1,\ldots,k+l+1}$.

We now apply Lemma~\ref{lemABC} to complete the proof. More explicitly, since the co-$k$-faces have the FAP with respect to their co-$l$-faces,
\[\begin{array}{rcl}
\Gamma(\mathcal{P})/N_{k+l+1}^{-}
&=&\langle\rho_{k+1},\ldots,\rho_{d-1}\rangle N_{k}^{-} / \hat{N}_{k+1,\ldots,k+l+1}N_{k}^{-} \\
&=&\langle\rho_{k+1},\ldots,\rho_{d-1}\rangle / \hat{N}_{k+1,\ldots,k+l+1}\\
&\cong&\langle\rho_{k+l+2},\ldots,\rho_{d-1}\rangle.
\end{array}\]
Thus $\mathcal{P}$ has the FAP with respect to its co-$(k+l+1)$-faces. 
\end{proof}
\medskip

The flat amalgamation property can be employed to construct flat regular polytopes $\mathcal{P}$ with preassigned faces $\mathcal{P}_{1}$ and co-faces $\mathcal{P}_{2}$.
Suppose $\mathcal{P}_{1}$ is a regular $c$-polytope and $\mathcal{P}_{2}$ a regular $d$-polytope such that the co-$k$-faces $\mathcal{K}$ of $\mathcal{P}_{1}$ are isomorphic to the $(c-k-1)$-faces of $\mathcal{P}_{2}$. Let $\Gamma(\mathcal{P}_{1}) = \langle\alpha_{0},\ldots,\alpha_{c-1}\rangle$, and $\Gamma(\mathcal{P}_{2}) = \langle\beta_{k+1},\ldots,\beta_{k+d}\rangle$, where the subscripts of the standard generators of $\Gamma(\mathcal{P}_{2})$ have been shifted by $k+1$. As in \cite[4F5]{McMSch2002}, consider the subgroup 
$\Gamma := \Gamma(\mathcal{P}_{1},\mathcal{P}_{2}) := \langle\rho_0,\ldots,\rho_{k+d}\rangle$ of the direct product $\Gamma(\mathcal{P}_{1}) \times \Gamma(\mathcal{P}_{2})$ where the generators $\rho_i$ are defined by
\begin{equation}
\label{flatamalgen}
\rho_{i} := \left\{  
\begin{array}{lcl}
(\alpha_{i},1) & \mbox{if} & 0 \leq i \leq k,  \\
(\alpha_{i},\beta_{i}) & \mbox{if} & k+1 \leq i \leq c-1,  \\
(1,\beta_{i}) & \mbox{if} & c \leq i \leq k+d.
\end{array}  \right.
\end{equation}
Then $\Gamma$ is called the $(k+1)$-{\it mix\/} of the groups $\Gamma(\mathcal{P}_{1})$ and $\Gamma(\mathcal{P}_{2})$, denoted $\Gamma(\mathcal{P}_{1})\!\diamond_{k+1}\!\Gamma(\mathcal{P}_{2})$. If $\Gamma$ happens to be a C-group (as in the situation described below), then the corresponding regular polytope is called the $(k+1)$-mix of the polytopes $\mathcal{P}_1$ and $\mathcal{P}_2$ denoted $\mathcal{P}_{1}\!\diamond_{k+1}\!\mathcal{P}_{2}$.
The specific nature of the generators immediately shows that $N_{k}^{-}$ and $N_{c}^{+}$ lie in $\Gamma(\mathcal{P}_{1})\times\{1\}$ or $\{1\}\times\Gamma(\mathcal{P}_{2})$, respectively, and thus commute at the level of elements.

The following theorem summarizes the properties of this mix construction (see \cite[Thm. 4F9]{McMSch2002}). 

\begin{theorem}
\label{amalprops}
Let $c,d \geq 2$, and let $0 \leq k \leq c-2$ and $k \geq c-d$.  Let $\mathcal{P}_{1}$
be a regular $c$-polytope with group $\Gamma(\mathcal{P}_{1}) = \langle\alpha_{0},\ldots,\alpha_{c-1}\rangle$, and let $\mathcal{P}_{2}$ be a regular $d$-polytope with $\Gamma(\mathcal{P}_{2}) = \langle\beta_{k+1},\ldots,\beta_{k+d}\rangle$ (with subscripts shifted by $k+1$). Suppose that the co-$k$-faces $\mathcal{K}$ of $\mathcal{P}_{1}$ are isomorphic to the $(c-k-1)$-faces of $\mathcal{P}_{2}$, and that $\mathcal{P}_{1}$ has the FAP with respect to its co-$k$-faces and $\mathcal{P}_{2}$ has the FAP with respect to its $(c-k-1)$-faces. Then the $(k+1)$-mix $\Gamma = \Gamma(\mathcal{P}_{1},\mathcal{P}_{2})$ of $\Gamma(\mathcal{P}_{1})$ and $\Gamma(\mathcal{P}_{2})$ is a string C-group of rank $k+d+1$. Let $\mathcal{P} := \mathcal{P}_{1} \diamond_{k+1} \mathcal{P}_{2}$ be the corresponding regular $(k+d+1)$-polytope with group $\Gamma(\mathcal{P}) = \Gamma$. Then $\mathcal{P}$ has the following properties:\\[.02in]
(a)\, The $c$-faces of $\mathcal{P}$ are isomorphic to $\mathcal{P}_{1}$, and the
co-$k$-faces are isomorphic to $\mathcal{P}_{2}$.  More generally, if $r \leq k$, \ 
$s \geq c$, \ $s-r \geq 4$, and $\{F_{-1},F_0,\ldots,F_{k+d+1}\}$ denotes the base flag of $\mathcal{P}$, then the section $F_{s}/F_{r}$ of $\mathcal{P}$ is
isomorphic to the regular $(s-r-1)$-polytope $\mathcal{Q} := \mathcal{Q}_{1} \diamond_{k-r}\mathcal{Q}_{2}$ which is constructed from the co-$r$-face $\mathcal{Q}_{1}$ of $\mathcal{P}_{1}$ and $(s-k-1)$-face $\mathcal{Q}_{2}$ of $\mathcal{P}_{2}$ in the same way as $\mathcal{P}$ is from $\mathcal{P}_{1}$ and $\mathcal{P}_{2}$.\\[.02in]
(b)\, $\mathcal{P}$ has the FAP with respect to both its co-$k$-faces and its
$c$-faces.  In particular,
\[  \Gamma(\mathcal{P}) \cong 
N_{k}^{-} \rtimes \Gamma(\mathcal{P}_{2}) \cong
N_{c}^{+} \rtimes \Gamma(\mathcal{P}_{1}) \cong 
(N_{k}^{-} \times N_{c}^{+}) \rtimes \Gamma(\mathcal{K}),  \]
where $N_{k}^{-}$ and $N_{c}^{+}$ are the normal closures of 
$\{\alpha_{0},\ldots,\alpha_{k}\}$ in $\Gamma(\mathcal{P}_{1})$ and 
$\{\beta_{c},\ldots,\beta_{k+d}\}$ in $\Gamma(\mathcal{P}_{2})$, respectively.\\[.02in]
(c)\, $\mathcal{P}$ is $(k,c)$-flat.
\end{theorem}

\section{Toroidal polyhedra of type $\{4,4\}$}
\label{44type}

The toroidal regular polyhedra $\{4,4\}_{(s,t)}$, with $s\geq 2$, $t=0$ or $s=t\geq 2$, have been extensively studied (see Coxeter \& Moser~\cite[Ch. 8]{CoxMos1980}). 
They can be derived from the square tessellation of the plane with vertex set $\mathbb{Z}^2$ as the quotient by the lattices spanned by the vectors $(s,t)$, $(-t,s)$, respectively. Here we briefly discuss the polyhedra $\{4,4\}_{(s,t)}$ whose automorphism groups are 2-groups.

\begin{lemma}
\label{lem44}
The regular polyhedra of type $\{4,4\}$ with 2-groups as automorphism groups are the polyhedra $\{4,4\}_{(2^e,0)}$ and $\{4,4\}_{(2^e,2^e)}$, $e\geq 1$, with group orders $2^{2e+3}$ and $2^{2e+4}$, respectively. 
\end{lemma}

\begin{proof}
A regular polyhedron $\{4,4\}_{(s,t)}$, with $s\geq 2$, $t=0$ or $s=t\geq 2$, has an automorphism group of order $8(s^2+t^2)$. For this group to be 2-group, $s^2+t^2$ must be a power of 2. Using simple arithmetic modulo 4, it is not difficult to see that each power of 2 can be written in exactly one way as a sum of two squares of non-negative numbers, up to interchanging the two summands:\ if $k=2e$ is even, then $2^k = (2^e)^2 +0^2$, and if $k=2e+1$ is odd, then $2^k = (2^e)^2 +(2^e)^2$. Thus, if the order of the group is $2^n$, $n\geq 4$, the equation $2^{n-3}=s^2+t^2$ gives the parameters $(s,t)=(2^e,0),(2^e,2^e)$ with $n=2e+3,\,2e+4$, respectively. This proves the lemma.
\end{proof}

For $n\geq 5$, we let $\{4,4\}^{(n)}$ denote the unique regular polyhedron of type $\{4,4\}$ with an automorphism group of order $2^n$. Thus, 
\begin{equation} 
\{4,4\}^{(n)}=\{4,4\}_{(2^e,0)}\mbox{ or }\{4,4\}_{(2^e,2^e)},
\end{equation} 
according as $n$ is odd, $n-3=2e$, or $n$ is even, $n-4=2e$. We similarly use the notation $[4,4]^{(n)}$, $[4,4]_{(2^e,0)}$ or $[4,4]_{(2^e,2^e)}$ for the respective automorphism groups. Each regular polyhedra $\{4,4\}^{(n)}$ is self-dual and has a polarity $\omega$ (duality of order 2) that fixes the base flag. If we let $[4,4]^{(n)} = \langle\alpha_0,\alpha_1,\alpha_2\rangle$, then conjugation by $\omega$ in the extended group (of all automorphisms and dualities) swaps the generators $\alpha_0$ and $\alpha_2$ while keeping the generator $\alpha_1$ fixed.

The polyhedra $\{4,4\}^{(n)}$ all have the flat amalgamation property with respect to both, the vertex-figures and the facets~\cite{Sch1988}. In particular, if $N(\alpha_0)$ denotes the normal closure of $\alpha_0$ in $[4,4]^{(n)}$, then $N(\alpha_0)$ has order $2^{n-3}$ and
\begin{equation}
\label{torfap1}
[4,4]^{(n)} = N(\alpha_0)\rtimes\langle\alpha_1,\alpha_2\rangle 
\cong N(\alpha_0) \rtimes D_4 .
\end{equation}
Similarly, 
\begin{equation}
\label{torfap2}
[4,4]^{(n)} = N(\alpha_2)\rtimes\langle\alpha_0,\alpha_1\rangle 
\cong N(\alpha_2)\rtimes D_4 ,
\end{equation}
where $N(\alpha_2)$ denotes the normal closure of $\alpha_2$ in $[4,4]^{(n)}$.
\medskip

\section{Regular polytopes with 2-groups}
\label{regpols2}
\medskip

In this section, we construct flat regular $d$-polytopes $\mathcal P$ whose automorphism groups are finite 2-groups and whose successive sections of rank 3 as in (\ref{rk3sections}) are toroidal maps of the kind described in the previous section. Any such polytope $\mathcal P$ can be assigned the generalized Schl\"afli symbol
\begin{equation}
\label{all4stypes}
\{\{4,4\}^{(n_3)},\{4,4\}^{(n_4)},\ldots,\{4,4\}^{(n_{d})}\},
\end{equation}
where here the toroidal polyhedron $\{4,4\}^{(n_j)}$ (whose automorphism group has order $2^{n_j}$) gives the isomorphism type of the section 
$F_{j}/F_{j-4}$ of $\mathcal P$ for $j=3,\ldots,d$.  We will show that any sequence of numbers $n_3,\ldots,n_d$, with $n_{j}\geq 5$ for $j=3,\ldots,d$, can be preassigned. 
\smallskip

Our result is based on the construction technique for flat regular polytopes described in Theorem~\ref{amalprops}. There are several ways in which this method can be employed to obtain polytopes of the desired type (\ref{all4stypes}) with automorphism groups that are 2-groups. The polytopes derived here have relatively small 2-groups as automorphism groups.

\begin{theorem}
\label{flat44thm}
Let $d\geq 3$, and let $n_3,\ldots,n_d\geq 5$. Then there exists a finite regular $d$-polytope $\mathcal P$ with the following properties:\\[.02in]
(a) $\mathcal P$ is of type $\{\{4,4\}^{(n_3)},\{4,4\}^{(n_4)},\ldots,\{4,4\}^{(n_{d})}\}$.\\[.02in]
(b) The automorphism group of $\mathcal P$ is a 2-group of order 
$2^{n_3+n_4+\ldots+n_d - 3(d-3)}$.\\[.02in]
(c) $\mathcal P$ is flat if $d\geq 4$; in fact, $\mathcal P$ is $(d-4,d-1)$-flat if $d\geq 4$.\\[.02in]
(d) $\mathcal P$ has the FAP with respect to its co-$(d-4)$-faces, its co-$(d-3)$-faces, and its facets.
\end{theorem}

\begin{proof}
We use Theorem~\ref{amalprops} and exploit the fact that the toroidal maps $\{4,4\}^{(n_j)}$ all have the flat amalgamation property with respect to both, the facets and the vertex-figures. 

Our proof is by induction on the rank $d$. For each dimension $d\geq 3$, we construct a regular polytope ${\mathcal P}_d$ with the desired properties. For $d=3$ there is nothing to prove: ${\mathcal P}_{3}=\{4,4\}^{(n_3)}$. For $d\geq 4$, the polytope ${\mathcal P}_{d}$ will be derived by amalgamating ${\mathcal P}_{d-1}$ and $\{4,4\}^{(n_d)}$ along the $2$-sections which are the co-$(d-4)$-face of ${\mathcal P}_{d-1}$ and the $2$-face of $\{4,4\}^{(n_d)}$, respectively. 

It is instructive to first consider the case $d=4$. Set ${\mathcal Q}_1:=\{4,4\}^{(n_3)}$ and 
${\mathcal Q}_2:=\{4,4\}^{(n_4)}$. Then ${\mathcal Q}_1$ has the FAP with respect to its vertex-figures, and ${\mathcal Q}_2$ has the FAP with respect to its 2-faces. In particular, if we let $\Gamma(\mathcal{Q}_1)=\langle\alpha_0,\alpha_1,\alpha_2\rangle$ and $\Gamma(\mathcal{Q}_2)=\langle\beta_1,\beta_2,\beta_3\rangle$, then 
\[\Gamma(\mathcal{Q}_1)=N(\alpha_0)\rtimes\langle\alpha_1,\alpha_2\rangle,\quad
\Gamma(\mathcal{Q}_2)=N(\beta_3)\rtimes\langle\beta_1,\beta_2\rangle,\]
where $\langle\alpha_1,\alpha_2\rangle\cong\langle\beta_1,\beta_2\rangle\cong D_4$ and $N(\alpha_0)$ and $N(\beta_3)$ are the normal closures of $\alpha_0$ in 
$\Gamma(\mathcal{Q}_1)$ and $\beta_3$ in $\Gamma(\mathcal{Q}_2)$, respectively. Here, the subgroups $N(\alpha_0)$ and $N(\beta_3)$ have orders $2^{n_{3}-3}$ and $2^{n_{4}-3}$, respectively. Now, by Theorem~\ref{amalprops} above (with $k=0$), there exists a $(0,3)$-flat regular 4-polytope, denoted $\mathcal{P}_4$,  such that the facets and vertex-figures of ${\mathcal P}_4$ are isomorphic to ${\mathcal Q}_1$ and ${\mathcal Q}_2$, respectively. Thus ${\mathcal P}_4$ is of the desired type $\{\{4,4\}^{(n_3)},\{4,4\}^{(n_4)}\}$. If we let $\Gamma({\mathcal P}_4)=\langle\rho_0,\ldots,\rho_3\rangle$, with generators $\rho_i$ as in (\ref{flatamalgen}), and let $N_{0}\,(\cong N(\alpha_0))$ and $N_{3}\,(\cong N(\beta_3))$ denote the normal closures of $\rho_0$ and $\rho_3$ in the facet subgroup $\langle\rho_0,\rho_1,\rho_2\rangle$ and vertex-figure subgroup $\langle\rho_1,\rho_2,\rho_3\rangle$, respectively, then $N_{0}$ and $N_{3}$ are also the normal closures of $\rho_0$ and $\rho_3$ in the full group $\Gamma({\mathcal P}_4)$ and
\[\Gamma({\mathcal P}_4)
\cong N_0\rtimes\Gamma(\mathcal{Q}_2)
\cong N_3\rtimes\Gamma(\mathcal{Q}_1)
\cong (N_0\times N_3)\rtimes D_4 ,\]
where the dihedral group $D_4$ is given by $\langle\rho_1,\rho_2\rangle$. Moreover, ${\mathcal P}_4$ has the FAP with respect to its vertex-figures and its facets. Note that $\Gamma({\mathcal P}_4)$ has order
\[2^{(n_{3}-3)+(n_{4}-3)+3}=2^{n_3+n_4-3}=2^{n_3+n_4-3(d-3)}.\]

In order for us to be able to proceed to the next rank via the outlined amalgamation process, we must verify that $\mathcal{P}_4$ also has the FAP with respect to its edge-figures. But this now follows from Lemma~\ref{lemher}(a). In fact, since $\mathcal{P}_4$ has the FAP with respect to its vertex-figures, and the vertex-figures of $\mathcal{P}_4$ have the FAP with respect to their vertex-figures, Lemma~\ref{lemher}(a) shows that $\mathcal{P}_4$ also has the FAP with respect to its edge-figures. In particular, 
\[\Gamma({\mathcal P}_4) = N_{0,1} \langle\rho_2,\rho_3\rangle 
\cong N_{0,1}\rtimes \langle\rho_2,\rho_3\rangle,\]
where $N_{0,1}$ is the normal closure of $\{\rho_0,\rho_1\}$ in $\Gamma({\mathcal P}_4)$. In this context, (\ref{herNs}) takes the form 
\begin{equation}
\label{n01n1hat}
N_{0,1}=N_0 \,\hat{N}_{1} \cong N_0 \rtimes \hat{N}_{1},
\end{equation}
where $\hat{N}_{1}$ is the normal closure of $\rho_1$ in the vertex-figure subgroup $\langle\rho_1,\rho_2,\rho_3\rangle$.

Note that the semidirect product in (\ref{n01n1hat}) is not a direct product; the elements $\rho_{1}\rho_{0}\rho_{1}$ of $N_0$ and $\rho_{2}\rho_{1}\rho_{2}$ of $\hat{N}_{1}$ do not commute.
In fact, if we assume the contrary and keep in mind that $\rho_1\rho_2$ has order 4, then the equation
$\rho_{2}\rho_{1}\rho_{2}\cdot\rho_{1}\rho_{0}\rho_{1} \cdot\rho_{2}\rho_{1}\rho_{2}
=\rho_{1}\rho_{0}\rho_{1}$ simplifies to $\rho_{2}\rho_{1}\rho_{2}\rho_{0}\rho_{2}\rho_{1}\rho_{2}=\rho_0$ and therefore $(\rho_1\rho_0)^2=1$, contradicting the order of $\rho_0 \rho_1$. 

For the inductive step of the construction, let $d\geq 5$ and suppose there exists a regular $(d-1)$-polytope ${\mathcal P}_{d-1}$ with the properties listed in the theorem. Set ${\mathcal Q}_1:=\mathcal{P}_{d-1}$ and ${\mathcal Q}_2:=\{4,4\}^{(n_d)}$. Then ${\mathcal Q}_1$ has the FAP with respect to its co-$(d-4)$-face, and ${\mathcal Q}_2$ has the FAP with respect to its $2$-faces. Now, by Theorem~\ref{amalprops} above, there exists a $(d-4,d-1)$-flat regular $d$-polytope, denoted $\mathcal{P}_d$,  such that the facets and co-$(d-4)$-faces of ${\mathcal P}_d$ are isomorphic to ${\mathcal Q}_1$ and ${\mathcal Q}_2$, respectively. If we let $\Gamma({\mathcal P}_d)=\langle\rho_0,\ldots,\rho_{d-1}\rangle$, then 
\[\Gamma({\mathcal P}_d)
\cong N_{0,1,\ldots,d-4}\rtimes\Gamma(\mathcal{Q}_2)
\cong N_{d-1}\rtimes\Gamma(\mathcal{Q}_1)
\cong (N_{0,1,\ldots,d-4}\times N_{d-1})\rtimes D_4 ,\]
where here $N_{0,1,\ldots,d-4}$ denotes the normal closure of $\{\rho_0,\rho_1,\ldots,\rho_{d-4}\}$ in the facet subgroup $\langle\rho_0,\ldots,\rho_{d-2}\rangle$, and $N_{d-1}$ is the normal closure of $\rho_{d-1}$ in the co-$(d-4)$-face subgroup $\langle\rho_{d-3},\rho_{d-2},\rho_{d-1}\rangle$. These subgroups $N_{0,1,\ldots,d-4}$ and $N_{d-1}$ are also the normal closures of $\{\rho_0,\rho_1,\ldots,\rho_{d-4}\}$ and $\rho_{d-1}$ in the full group $\Gamma({\mathcal P}_d)$, respectively. As $\mathcal{Q}_{1}=\mathcal{P}_{d-1}$, the order of the automorphism group is given by 
\[|\Gamma({\mathcal P}_d)| = |N_{d-1}|\cdot |\Gamma(\mathcal{P}_{d-1})|
=2^{n_{d}-3}\cdot 2^{n_3+\ldots +n_{d-1} -3(d-4)} = 2^{n_3+\ldots +n_{d} -3(d-3)}.\]
By Theorem~\ref{amalprops}, ${\mathcal P}_d$ also has the FAP with respect to both, its co-$(d-4)$-faces and its facets.

It remains to prove that ${\mathcal P}_d$ also has the FAP with respect to its co-$(d-3)$-faces. We can proceed as for rank 4 and appeal to Lemma~\ref{lemher}. Then it follows that $\mathcal{P}_d$ has the FAP with respect to its co-$(d-3)$-faces and 
\[\Gamma({\mathcal P}_d) = N_{0,\ldots,d-3}\, \langle\rho_{d-2},\rho_{d-1}\rangle 
\cong N_{0,\ldots,d-3}\rtimes \langle\rho_{d-2},\rho_{d-1}\rangle,\]
where $N_{0,\ldots,d-3}$ is the normal closure of $\{\rho_0,\ldots,\rho_{d-3}\}$ in $\Gamma({\mathcal P}_d)$. Here, (\ref{herNs}) takes the form 
\begin{equation}
\label{theNs}
N_{0,\ldots,d-3}=N_{0,\ldots,d-4} \,\hat{N}_{d-3}\cong N_{0,\ldots,d-4} \rtimes \hat{N}_{d-3},
\end{equation}
where $\hat{N}_{d-3}$ is the normal closure of $\rho_{d-3}$ in the co-$(d-4)$-face subgroup $\langle\rho_{d-3},\rho_{d-2},\rho_{d-1}\rangle$.
Note that the semidirect product in (\ref{theNs}) is not direct, as the elements $\rho_{d-3}\rho_{d-4}\rho_{d-3}$ of $N_{0,\ldots,d-4}$ and $\rho_{d-2}\rho_{d-3}\rho_{d-2}$ of $\hat{N}_{d-3}$ do not commute, for similar reasons as before.

Thus the regular $d$-polytope ${\mathcal P}_d$ has the desired properties, and the inductive proof of the theorem is complete.
\end{proof}
\smallskip

\begin{remark}
The group of the regular $d$-polytope $\mathcal{P}_d$ of Theorem~\ref{flat44thm} is an ``iterated semi-direct product." Recall that equation (\ref{theNs}) in the above proof immediately gives
\[ N_{0,\ldots,d-3}=N_{0,\ldots,d-4} \rtimes \hat{N}_{d-3} \;\; (d\geq 4).\]
Iterating this down in rank and using (\ref{n01n1hat}), we obtain sequences of semi-direct products, beginning with
\[ N_{0,\ldots,d-3}=(N_{0,\ldots,d-5} \rtimes \hat{N}_{d-4})\rtimes \hat{N}_{d-3},\]
and ending with
\begin{equation}
\label{iteratedN}
N_{0,\ldots,d-3}=((((N_{0}\rtimes \hat{N}_{1})\rtimes \hat{N}_{2})\rtimes\,\ldots\,)\rtimes\hat{N}_{d-4})\rtimes\hat{N}_{d-3}\quad(d\geq 4),
\end{equation}
where $N_0$ is the normal closure of $\rho_0$ in $\langle\rho_{0},\rho_{1},\rho_{2}\rangle$ and each 
$\hat{N}_j$ for $j=1,\ldots,d-3$ is the normal closure of $\rho_{j}$ in the subgroup $\langle\rho_{j},\rho_{j+1},\rho_{j+2}\rangle$.  Here, the factor $N_0$ has order $2^{{n_3}-3}$ and each factor $\hat{N}_{j}$ has order $2^{n_{j+3}-3}$. Thus $N_{0,\ldots,d-3}$ is an iterated semi-direct product of order $2^{n_{3}+\ldots+n_{d} -3(d-2)}$. Further, since  
$\Gamma(\mathcal{P}_d)\cong N_{0,\ldots,d-3}\rtimes D_4$ with $D_4\cong\langle\rho_{d-2},\rho_{d-1}\rangle$, we also have the iterated semidirect product decomposition for the full group, 
\begin{equation}
\label{itergroup}
\Gamma(\mathcal{P}_d)
\cong (((((N_{0}\rtimes \hat{N}_{1})\rtimes \hat{N}_{2})\rtimes\,\ldots\,)\rtimes\hat{N}_{d-4})\rtimes\hat{N}_{d-3})\rtimes D_4\quad (d\geq 3),
\end{equation}
thus confirming $2^{n_{3}+\ldots+n_{d} -3(d-3)}$ as the order of $\Gamma(\mathcal{P}_d)$. Note here that we could have replaced $N_0$ by $\hat{N}_0$, with $\hat{N}_0$ defined just like $\hat{N}_j$ with $j=0$.
\end{remark}
\medskip

When $d=4$ some of the smaller polytopes $\mathcal{P}$ of Theorem~\ref{flat44thm} are actually universal, that is, isomorphic to a regular polytope $\{\{4,4\}_{(p,q)},\{4,4\}_{(r,s)}\}$ (see \cite[p.~371]{McMSch2002}). For example, this is the case when the facet $\mathcal{P}_1$ or vertex-figure $\mathcal{P}_2$ is flat, that is, $(p,q)=(2,0)$ or $(r,s)=(2,0)$, respectively. In this case, up to duality, the corresponding polytopes $\mathcal{P}$ are 
\[ \{\{4,4\}_{(2,0)},\{4,4\}_{(2^e,0)}\},\;\; e\geq 1,\] 
with group $(C_2\times C_2)\rtimes [4,4]_{(2,0)}$ of order 128 for $e=1$ and 
$(D_{2^{e-1}}\times D_{2^{e-1}})\rtimes [4,4]_{(2,0)}$ of order $2^{2e+5}$ for $e\geq 2$; as well as 
\[\{\{4,4\}_{(2,0)},\{4,4\}_{(2^e,2^e)}\},\;\; e\geq 1,\] 
with group $(D_{2^e}\times D_{2^e}\times C_2\times C_2)\rtimes (C_2\times C_2)$ of order $2^{2e+6}$. Note that when $n_{3}=n_{4}=6$, the polytope $\mathcal{P}$ of Theorem~\ref{flat44thm} is not universal; in this case the universal polytope
\[ \{\{4,4\}_{(2,2)},\{4,4\}_{(2,2)}\}\] 
is a 2-fold covering of $\mathcal{P}$, with group $C_2^4\rtimes [4,4]_{(2,2)}$ of order 1024, twice the order of the group of $\mathcal{P}$.
\smallskip

	Now let us determine a presentation for the automorphism groups of the polytopes of Theorem \ref{flat44thm}. 
	
	\begin{theorem}
	The automorphism group of the polytope $\calP$ of type 
	\[ \{ \{4,4\}^{(n_3)}, \ldots, \{4,4\}^{(n_d)} \} \]	
	guaranteed by Theorem \ref{flat44thm} is the quotient of the universal polytope of that type by the extra relators
	\[ [(\rho_0 \rho_1)^2, (\rho_2 \rho_3)^2], \ldots, [(\rho_{d-4} \rho_{d-3})^2, (\rho_{d-2} \rho_{d-1})^2]. \]
	\end{theorem}		
	
	\begin{proof}
	Let $\G(n_3, \ldots, n_d)$ be the quotient of 
	\[ \G(\{ \{4,4\}^{(n_3)}, \ldots, \{4, 4\}^{(n_d)} \}) \]
	by the extra relators  $[(\rho_0 \rho_1)^2, (\rho_2 \rho_3)^2], \ldots, [(\rho_{d-4} \rho_{d-3})^2, (\rho_{d-2} \rho_{d-1})^2]$. Consider the 4-faces of $\calP$, which are built by a flat amalgamation of $\{4,4\}^{(n_3)}$ with $\{4,4\}^{(n_4)}$. By construction, $(\rho_0 \rho_1)^2$ commutes with $(\rho_2 \rho_3)^2$ since $(\rho_0 \rho_1)^2 = ( (\alpha_0 \alpha_1)^2, 1)$ and $(\rho_2 \rho_3)^2 = (1, (\beta_2 \beta_3)^2)$. Similarly, since $\calP$ is built up iteratively by flat amalgamations, we see that $(\rho_i \rho_{i+1})^2$ commutes with $(\rho_{i+2} \rho_{i+3})^2$ for each $0 \leq i \leq d-4$. Thus, $\G(\calP)$ is a quotient of $\G(n_3, \ldots, n_d)$. To prove that $\G(\calP)$ is equal to this group, it suffices to show they have the same order.
	
Clearly, $\G(n_3, \ldots, n_d)$ covers 
	\[ \G(n_3, \ldots, n_{i-1}, 5, n_{i+1}, \ldots, n_d). \]
	The kernel $K$ of this covering is the normal closure of $\langle (\rho_{i-3} \rho_{i-2} \rho_{i-1} \rho_{i-2})^2 \rangle$. For simplicity's sake, let us assume $i = 4$, so that $K$ is the normal closure of $\langle (\rho_1 \rho_2 \rho_3 \rho_2)^2 \rangle$. Clearly then, $K$ contains
	\[ H = \langle (\rho_1 \rho_2 \rho_3 \rho_2)^2, (\rho_3 \rho_2 \rho_1 \rho_2)^2 \rangle. \]
	Now, the generators of $H$ correspond to perpendicular translations in the symmetry group of $\{4,4\}^{(n_4)}$, and so they commute and have trivial intersection, implying that $|H| = |(\rho_1 \rho_2 \rho_3 \rho_2)^2| \cdot |(\rho_3 \rho_2 \rho_1 \rho_2)^2| = 2^{n_4-5}$. It is easy to see that $H$ is normalized by $\rho_1$, $\rho_2$, and $\rho_3$. Furthermore, clearly $H$ is normalized by $\rho_j$ with $j \geq 5$. Let us show that $\rho_4$ commutes with the generators of $H$. Note that, since $(\rho_3 \rho_4)^2$ commutes with $(\rho_1 \rho_2)^2$ and $\rho_1$, it follows that $(\rho_3 \rho_4)^2$ commutes with $\rho_2 \rho_1 \rho_2$, and it then follows that $\rho_3 \rho_4 \rho_3$ commutes with $\rho_2 \rho_1 \rho_2$, then, suppressing the $\rho$s and working with subscripts only:
	\begin{align*}
	4~1232~1232~4~2321~2321 &= 4~1232~123~4~321~2321 \\
	&= 4123~212~343~212~321 \\
	&= 4123~343~321 \\
	&= 412421 \\
	&= e	
	\end{align*}
	A similar calculation shows that $\rho_4$ commutes with $\rho_3 \rho_2 \rho_1 \rho_2$, and that $\rho_0$ commutes with both generators of $H$. It follows that $H$ is normal, so $K = H$. Thus, no longer assuming that $i = 4$, we have
	\[ |\G(n_3, \ldots, n_d)| = |\G(n_3, \ldots, n_{i-1}, 5, n_{i+1}, \ldots, n_d)| \cdot |H|	= 2^{n_i-5} |\G(n_3, \ldots, n_{i-1}, 5, n_{i+1}, \ldots, n_d)|. \]
	Then by repeatedly taking quotients until every $n_j$ is 5, we get
	\[ |\G(n_3, \ldots, n_d)| = 2^{n_3 + \cdots + n_d - 5d + 10} |\G(5, \ldots, 5)|. \]
	Now consider $\G(5, \ldots, 5)$. In this group, we have $(\rho_i \rho_{i+1} \rho_{i+2} \rho_{i+1})^2 = 1$ for $0 \leq i \leq d-3$. Then
	\begin{align*}
	\rho_{i+2} (\rho_i \rho_{i+1})^2 \rho_{i+2} (\rho_{i+1} \rho_i)^2 &= \rho_i \rho_{i+2} \rho_{i+1} \rho_i \rho_{i+1} \rho_{i+2} \rho_{i+1} \rho_i \rho_{i+1} \rho_i \\
	&= \rho_i (\rho_{i+2} \rho_{i+1} \rho_i \rho_{i+1})^2 \rho_i \\
	&= 1,
	\end{align*}
	so $(\rho_i \rho_{i+1})^2$ commutes with $\rho_{i+2}$. A similar calculation shows that it commutes with $\rho_{i-1}$ (assuming $i \geq 1$), and it clearly commutes with every other $\rho_j$, so it is central. Then $\G(5, \ldots, 5)$ is a quotient of the group in \cite[Definition 5.1]{Con2013}, which has order $2^{2d-1}$ by \cite[Theorem 5.2]{Con2013}. On the other hand, every $d$-polytope of type $\{4, \ldots, 4\}$ has at least $2^{2d-1}$ flags by \cite[Theorem 3.2]{Con2013}, and so it follows that $\G(5, \ldots, 5)$ is precisely this group of order $2^{2d-1}$. Thus,
	\[ |\G(n_3, \ldots, n_d)| = 2^{n_3 + \cdots + n_d - 5d + 10 + 2d - 1} = 2^{n_3 + \cdots + n_d - 3(d-3)}, \]
	which is the order of $\G(\calP)$.
	\end{proof}

The amalgamation technique of Theorem~\ref{flat44thm} of concatenating regular polyhedra or polytopes as successive sections of regular polytopes of higher rank works in more general contexts than presented. Our focus here is on polytopes whose automorphism groups are 2-groups.

\section{Alternating semiregular polytopes with 2-groups}
\label{semiregpols2}

Semiregular polytopes are also a source of highly symmetric polytopes whose automorphism groups are 2-groups. Recall that an alternating semiregular $d$-polytope $\mathcal S$ has two kinds of regular facets $\mathcal P$ and $\mathcal Q$ occurring in alternating fashion around any $(d-3)$-face in $\mathcal S$. 

Alternating semiregular polytopes are closely related to tail-triangle C-groups (Monson \& Schulte~\cite{MonSch2012,MonSch2020}), that is, C-groups with an underlying tail-triangle diagram of the form 
\medskip
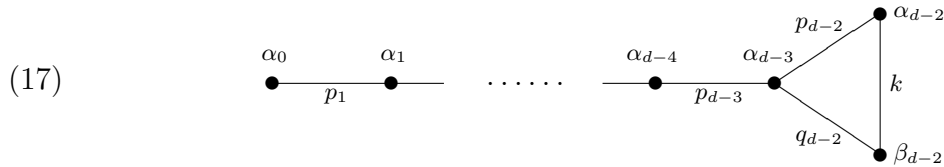
\begin{equation}
\label{tailtrianglegroup}
\begin{picture}(210,38)
\put(100,2){
\multiput(15,0)(45,0){2}{\circle*{5}}
\multiput(100,-27)(0,53.3){2}{\circle*{5}}
\multiput(-130,0)(45,0){2}{\circle*{5}}
\put(-5,0){\line(1,0){20}}
\put(-85,0){\line(1,0){20}}
\put(-130,0){\line(1,0){45}}
\put(15,0){\line(1,0){45}}
\put(60,0){\line(3,2){40}}
\put(60,0){\line(3,-2){40}}
\put(100,-27){\line(0,1){53.3}}
\put(4,9){\scriptsize $\alpha_{d-4}$}
\put(48,9){\scriptsize $\alpha_{d-3}$}
\put(30,-6){\scriptsize $p_{d-3}$}
\put(-134,9){\scriptsize $\alpha_{0}$}
\put(-89,9){\scriptsize $\alpha_{1}$}
\put(-110,-6){\scriptsize $p_1$}
\put(-49,-0.5){$\ldots\ldots$}
\put(68,21){\scriptsize $p_{d-2}$}
\put(68,-22){\scriptsize $q_{d-2}$}
\put(105,-29){\scriptsize $\beta_{d-2}$}
\put(105,25){\scriptsize $\alpha_{d-2}$}
\put(103.5,-3){\scriptsize $k$}}
\end{picture}
\end{equation}\medskip\bigskip

\noindent
Every $d$-generator tail-triangle C-group $\Gamma$ as in (\ref{tailtrianglegroup}) gives rise in a natural way to an alternating semiregular $d$-polytope $\mathcal S$. We will exploit some constructions of \cite{MonSch2012,MonSch2020} to derive polytopes with 2-groups.  

We first look at the case $d=3$, specifically the medial construction. When $d=3$ the tail is missing and the diagram is a triangle, possibly with $k=2$ so that the diagram is linear. 
\bigskip

\noindent
{\it Medials of regular polyhedra}
\medskip

When $d=3$ and $k=2$ the diagram for the C-group $\Gamma$ in (\ref{tailtrianglegroup}) is linear. In this case $\Gamma$ is a string C-group and thus gives rise to a regular polyhedron $\mathcal{K}$. On the other hand, viewing the diagram as a triangle diagram (with $k=2$), the same group $\Gamma$  also determines a semiregular polyhedron $\mathcal S$. In Coxeter's notation~\cite{Cox1973}, $\mathcal S$ is derived from $\Gamma$ by applying Wythoff's construction with the middle node of the (linear) diagram ringed. More explicitly, $\mathcal S$ is just the {\it medial\/} of the polyhedron $\mathcal{K}$ (see \cite{PisSer2013}). The medial $\mathcal S$ of $\mathcal K$ can be drawn on the same surface as $\mathcal K$ (if the faces and vertex-figures of $\mathcal K$ are finite):\ the vertices of $\mathcal S$ are the ``edge midpoints" of $\mathcal K$ (arbitrarily chosen relative interior points of the edges, one point per edge); the edges of $\mathcal S$ connect the edge midpoints on adjacent edges of a face of $\mathcal K$; and the faces of $\mathcal S$ are formed by the edge cycles that either pass through the edge midpoints of the edges of a face of $\mathcal K$ (so that the new face is ``inscribed" in the old), or the edge midpoints on successive edges around a vertex of $\mathcal{K}$ (so that, in a sense, the old vertex is being cut off). For example, if $\mathcal K$ is the cube $\{4,3\}$, then $\mathcal S$ is the cuboctahedron. If $\mathcal{K}$ is of type $\{p,q\}$ (with $p=p_{d-2}$ and $q=q_{d-2}$), then the medial $\mathcal S$ has $p$-gons and $q$-gons as faces, with two of each kind alternating around each vertex of $\mathcal S$. Note that the semiregular polyhedron $\mathcal S$ is regular if and only if $\mathcal K$ is self-dual. In this case, $\Gamma(\mathcal{S})=\Gamma(\mathcal{K})\rtimes C_2$. In all other cases, $\Gamma(\mathcal{S})=\Gamma(\mathcal{K})$.

Via the medial construction, each regular polyhedron $\mathcal K$ with a 2-group as automorphism group gives rise to a semiregular polyhedron $\mathcal S$, with the same automorphism group if $\mathcal{K}$ is not self-dual, but with an automorphism group twice as large as the original group if $\mathcal K$ is self-dual. If $\mathcal{K}$ is of type $\{2^k,2^l\}$, then two  $2^k$-gons and two $2^l$-gons alternate around every vertex of $\mathcal{S}$.
\bigskip

\noindent
{\it Small alternating semiregular polytopes}
\medskip

Recall that $\{4,4\}^{(n)}$ denotes the unique regular polyhedron of type $\{4,4\}$ with a group of order $2^n$. In particular, $\{4,4\}^{(n)}$ is $\{4,4\}_{(2^e,0)}$ or $\{4,4\}_{(2^e,2^e)}$, according as $n$ is odd, $n-3=2e$, or $n$ is even, $n-4=2e$.

Just as for regular polytopes, we can preassign any combination of toroidal maps of type $\{4,4\}^{(n)}$ to occur as the collection of (regular) sections of rank 3 in a small alternating semiregular polytope whose automorphism group is of 2-power order. We again abuse standard Schl\"afli symbol notation and assign any such $d$-polytope $\mathcal S$, $d\geq 4$, the \textit{generalized Schl\"afli symbol}
\begin{equation}
\label{all4stypessemi}
\left\{\!\!\begin{array}{rl}
\{4,4\}^{(n_3)},\ldots,\{4,4\}^{(n_{d-3})},\{4,4\}^{(n_{d-2})},
&\!\!\!\!\!\!\begin{array}{l}
\{4,4\}^{(n_{d-1})}\\[.1in]
\{4,4\}^{(m_{d-1})}
\end{array}
\end{array}\!\!\!\!\!\right\},
\end{equation}
with the convention that when $d=4$ the symbol just reduces to  
\[\left\{\!\!\begin{array}{l}
\{4,4\}^{(n_{d-1})}\\[.1in]
\{4,4\}^{(m_{d-1})}
\end{array}\!\!\!\right\}.\]
Then the two kinds of regular facets $\mathcal P$ and $\mathcal Q$ of $\mathcal{S}$ have symbols 
\[\{\{4,4\}^{(n_3)},\ldots,\{4,4\}^{(n_{d-3})},\{4,4\}^{(n_{d-2})},\{4,4\}^{(n_{d-1})}\}\]
and 
\[\{\{4,4\}^{(n_3)},\ldots,\{4,4\}^{(n_{d-3})},\{4,4\}^{(n_{d-2})},\{4,4\}^{(m_{d-1})}\},\]
respectively, and for $d\geq 5$ meet along common $(d-2)$-faces with the symbol
\[\{\{4,4\}^{(n_3)},\ldots,\{4,4\}^{(n_{d-3})},\{4,4\}^{(n_{d-2})}\}\]
and for $d=4$ along 2-faces, $\{4\}$. 

In an alternating semiregular $d$-polytope $\mathcal S$, all sections of rank $3$, but one, are sections of facets of $\mathcal{S}$ and are represented in the above symbol. The exception is the co-$(d-4)$-faces, which usually are not regular. These are not part of the Schl\"afli symbol and are not preassigned. If $S$ can be obtained from a tail-triangle C-group, then the co-$(d-4)$-faces are medials of regular polyhedra.

The following theorem employs the construction technique for small alternating semiregular polytopes described in \cite[Thm. 5.3]{MonSch2020}, along with the results of Theorem~\ref{flat44thm} above.

\begin{theorem}
\label{44semiregularthm}
Let $d\geq 4$, and let $n_3,\ldots,n_{d-1}\geq 5$ and $m_{d-1}\geq 5$. Then there exists a finite alternating semiregular $d$-polytope $\mathcal P$ with the following properties:\\[.02in]
(a) $\mathcal P$ is of type 
\[\left\{\!\!\begin{array}{rl}
\{4,4\}^{(n_3)},\ldots,\{4,4\}^{(n_{d-3})},\{4,4\}^{(n_{d-2})},
&\!\!\!\!\!\!\begin{array}{l}
\{4,4\}^{(n_{d-1})}\\[.1in]
\{4,4\}^{(m_{d-1})}
\end{array}
\end{array}\!\!\!\!\!\right\},\]
with appropriate interpretation as above if $d=4$.\\[.02in]
(b) The automorphism group of $\mathcal P$ is a 2-group and has order 
\[2^{n_3+n_4+\ldots+n_{d-2} +n_{d-1}+m_{d-1}-3(d-3)}\,\mbox{ if } n_{d-1}\neq m_{d-1}\]
or twice this order if $n_{d-1}=m_{d-1}$.
\end{theorem}

\begin{proof}
By our Theorem~\ref{flat44thm} applied to the common subsequence $n_3,\ldots,,n_{d-2}$ of  the two sequences $n_3,\ldots,n_{d-2},n_{d-1}$ and $n_3,\ldots,n_{d-2},m_{d-1}$, there exists a small regular $(d-2)$-polytope $\mathcal K$ of type  
\[\{\{4,4\}^{(n_3)},\ldots,\{4,4\}^{(n_{d-2})}\},\] 
with the properties listed in Theorem~\ref{flat44thm}. Moreover, again by Theorem~\ref{flat44thm}, now applied to the sequences  $n_3,\ldots,n_{d-2},n_{d-1}$ and $n_3,\ldots,n_{d-2},m_{d-1}$, there also exist small regular $(d-1)$-polytopes $\mathcal P$ and $\mathcal Q$ of types 
\[\{\{4,4\}^{(n_3)},\ldots,\{4,4\}^{(n_{d-3})},\{4,4\}^{(n_{d-2})},\{4,4\}^{(n_{d-1})}\}\]
and 
\[\{\{4,4\}^{(n_3)},\ldots,\{4,4\}^{(n_{d-3})},\{4,4\}^{(n_{d-2})},\{4,4\}^{(m_{d-1})}\},\]
respectively. In particular, by the inductive nature of our construction, $\mathcal{K}$ is a facet of both $\mathcal P$ and $\mathcal Q$; and by part (d) of Theorem~\ref{flat44thm}, both $\mathcal P$ and $\mathcal Q$ also have the FAP with respect to their facets. If $m_{d-1}=n_{d-1}$, then $\mathcal P$ is isomorphic to $\mathcal Q$. 

Now, as $\mathcal P$ and $\mathcal Q$ both have the FAP with respect to their common facet $\mathcal K$, we can appeal to \cite[Thm.~5.3]{MonSch2020} to derive an alternating semiregular $d$-polytope $\mathcal S$ with two copies of $\mathcal P$ and $\mathcal Q$ occurring alternately around each co-$(d-3)$-face of $\mathcal S$. The polytope $\mathcal S$ is obtained from a group $\Gamma$ with generators $\alpha_0,\ldots,\alpha_{d-3},\alpha_{d-2},\beta_{d-2}$ as in (\ref{tailtrianglegroup}), where here $\Gamma(\mathcal{P})$ and $\Gamma(\mathcal{Q})$ are identified with $\langle\alpha_0,\ldots,\alpha_{d-3},\alpha_{d-2}\rangle$ and 
$\langle\alpha_0,\ldots,\alpha_{d-3},\beta_{d-2}\rangle$, respectively. If we let $N_{d-2}(\mathcal{P})$ and $N_{d-2}(\mathcal{Q})$, respectively, denote the normal closures of $\alpha_{d-2}$ and $\beta_{d-2}$ in $\Gamma(\mathcal{P})$ and $\Gamma(\mathcal{Q})$, then 
\[\Gamma(\mathcal{P})\cong N_{d-2}(\mathcal{P})\rtimes \Gamma(\mathcal{K}),\;\;\,
\Gamma(\mathcal{Q})\cong N_{d-2}(\mathcal{Q})\rtimes \Gamma(\mathcal{K}).\]
The tail-triangle C-group $\Gamma$ from which the polytope~$\mathcal S$ is constructed is given by
\[\Gamma\;\,\cong\;\, (N_{d-2}(\mathcal{P}) \times N_{d-2}(\mathcal{Q}))\,\rtimes\,\Gamma(\mathcal{K}).\]
Note that $\mathcal S$ is a regular polytope if and only if $m_{d-1}=n_{d-1}$. In this case, $\Gamma(\mathcal{S})\cong\Gamma\rtimes C_2$. In all other cases, $\Gamma(\mathcal{S})=\Gamma$.

By Theorem~\ref{flat44thm}, $\Gamma(\mathcal{K})$ is a 2-group of order $2^{n_3+n_4+\ldots+n_{d-2} - 3(d-5)}$.  Further, from the proof of Theorem~\ref{flat44thm} we know that $N_{d-2}(\mathcal{P})$ and $N_{d-2}(\mathcal{Q})$ have orders $2^{n_{d-1}-3}$ and 
$2^{m_{d-1}-3}$, respectively. It follows that $\Gamma$ is a 2-group of order 
\[2^{n_3+n_4+\ldots+n_{d-2} +n_{d-1}+m_{d-1}-3(d-3)}.\]
This is also the order of $\Gamma(\mathcal{S})$ if $m_{d-1}\neq n_{d-1}$. However, if $m_{d-1}= n_{d-1}$, the order is twice as big. This completes the proof.
\end{proof}

\begin{remark}
The numbers of vertices and facets of the alternating semiregular $d$-polytope $\mathcal S$ of Theorem~\ref{44semiregularthm} are given by
\[ f_{0}(\mathcal{S}) = \frac{f_{0}(\mathcal{P})f_{0}(\mathcal{Q})}{f_{0}(\mathcal{K})},\;\;\;\, 
f_{d-1}(\mathcal{S}) = f_{d-2}(\mathcal{P})+f_{d-2}(\mathcal{Q}) ,\]
where $f_0$ and $f_{d-1}$ denote the numbers of vertices and facets of the corresponding polytope, respectively. The polytope $\mathcal S$ is not flat in general, unlike in the regular case. Flatness would force all facets to have the same number of vertices, but $\mathcal{P}$ and $\mathcal Q$ may have different numbers. The co-$(d-4)$-faces of $\mathcal{S}$ are isomorphic to $\{4,4\}_{(2,2)}=\{4,4\}^{(6)}$, which is the medial of $\{4,4\}_{(2,0)}=\{4,4\}^{(5)}$, and thus are also regular. 
\end{remark}

\section{Polytopes with infinite 2-groups}
\label{inf2group}

Call a finitely generated infinite group $G$ an {\em infinite 2-group} if every element of $G$ of finite order has a 2-power order. Note that a finitely generated infinite group $G$ is an infinite 2-group if and only if every finite subgroup of $G$ is a finite 2-group. 

Recall that a (regular) {\em extension\/} of a regular $d$-polytope $\mathcal K$ is a regular $(d+1)$-polytope with facets isomorphic to $\mathcal K$. Given a regular $d$-polytope $\mathcal K$ there exists a regular $(d+1)$-polytope, called the {\em free extension} of $\mathcal{K}$, that covers every regular extension of $\mathcal K$ (see \cite[4D4]{McMSch2002}). Our goal is to show that the free extension of a regular $d$-polytope $\mathcal{K}$ that has a finite or an infinite $2$-group, is a regular $(d+1)$-polytope $\mathcal P$ with an infinite 2-group. We exploit the fact that the automorphism group of the free extension is isomorphic to a certain amalgamated free product. 

Suppose $\mathcal K$ is a regular $d$-polytope with group $\Gamma(\mathcal{K})=\langle\rho_{0},\ldots,\rho_{d-1}\rangle$. Let $\mathcal F$ denote the facet of $\mathcal K$ so that $\Gamma(\mathcal{F})=\langle\rho_{0},\ldots,\rho_{d-2}\rangle$, and write $\rho_d$ for the generator of a cyclic group $C_2$. Then the group of $\mathcal P$ is given by the amalgamated free product 
\begin{equation}
\label{amalreg}
\begin{array}{rcl}
\Gamma(\mathcal{P}) &=& 
\Gamma(\mathcal{K})\; {\mathop{\ast}}_{\,\Gamma(\mathcal{F})} \;(\Gamma(\mathcal{F}) \times C_{2}),\\[.05in]
&=&\langle\rho_{0},\ldots,\rho_{d-1}\rangle\;{\mathop{\ast}}_{\langle\rho_{0},\ldots,\rho_{d-2}\rangle}\; \langle\rho_{0},\ldots,\rho_{d-2},\rho_d\rangle,
\end{array}
\end{equation}
where $\rho_d$ generates the factor $C_2$ and, for $i = 0,\ldots,d-2$, the generators $\rho_{i}$ of the two copies of $\Gamma(\mathcal{F})$ in $\Gamma(\mathcal{K})$ and $\Gamma(\mathcal{F}) \times C_{2}$ are identified in the obvious way. 
\smallskip

We now have the following result.

\begin{theorem}
\label{freeextreg}
Let $\mathcal K$ be a regular $d$-polytope whose automorphism group is a finite or an infinite 2-group. Then the free extension $\mathcal P$ of $\mathcal K$ is a regular $(d+1)$-polytope with an infinite 2-group as automorphism group.
\end{theorem}

\begin{proof}
By (\ref{amalreg}), $\Gamma(\mathcal{P})$ is the free product of $\Gamma(\mathcal{K})$ and $\Gamma(\mathcal{F}) \times C_{2}$, amalgamated along their common subgroups $\Gamma(\mathcal{F})$. In particular, $\Gamma(\mathcal{P})$ is generated by $\rho_0,\ldots,\rho_{d-1},\rho_d$. By the Torsion Theorem for amalgamated free products (see \cite{LynSch1977}), each element of $\Gamma(\mathcal{P})$ of finite order is conjugate to an element of finite order in one of the two factors, $\Gamma(\mathcal{K})$ or $\Gamma(\mathcal{F})\times C_2$. Since $\Gamma(\mathcal{K})$ is a finite or an infinite 2-group, its facet subgroup $\Gamma(\mathcal{F})$ and thus $\Gamma(\mathcal{F})\times C_2$ are also finite or infinite 2-groups. It follows that every element of finite order in $\Gamma(\mathcal{P})$ is conjugate to an element of 2-power order in the factors $\Gamma(\mathcal{K})$ or $\Gamma(\mathcal{F})\times C_2$, and thus itself must have 2-power order.
\end{proof}
\medskip

A similar theorem also holds for free extensions of chiral polytopes. Recall that the automorphism group of a chiral $d$-polytope $\mathcal K$ has two flag orbits such that any two adjacent flags are in different orbits. If $\mathcal K$ is chiral, then $\Gamma(\mathcal{K})$ has two sets of distinguished generators $\sigma_1,\ldots,\sigma_{d-1}$ and $\sigma_1^{-1},\sigma_{2}\sigma_{1}^{2},\sigma_{3},\ldots,\sigma_{d-1}$, respectively, the first determined by the base flag and the second by the $0$-adjacent flag of the base flag (see \cite{Pel2025,SchWei1991}). These sets of generators are not conjugate in $\Gamma(\mathcal{K})$ and represent the two enantiomorphic (mirror image) forms of $\mathcal{K}$; in a sense, a {\it left\/} and {\it right\/} version of the same underlying polytope. An {\it oriented\/} chiral polytope consists of a chiral polytope together with one of these two sets of distinguished generators. 

A {\em chiral extension\/} of an oriented chiral $d$-polytope $\mathcal K$ is an oriented chiral $(d+1)$-polytope with (oriented) facets isomorphic to $\mathcal K$. Note that for $\mathcal K$ to admit a chiral extension, the facets $\mathcal F$ of $\mathcal K$ must be directly regular, meaning that their own rotation subgroup $\Gamma^{+}(\mathcal{F})$ has index 2 in their full automorphism group $\Gamma(\mathcal{F})$. It is known that given an oriented chiral $d$-polytope $\mathcal K$ with (necessarily) directly regular facets $\mathcal F$, there exists an oriented chiral $(d+1)$-polytope $\mathcal P$, called the {\em free extension} of $\mathcal{K}$, that covers every chiral extension of $\mathcal K$ (see Schulte \& Weiss~\cite{SchWei1995}). The group of $\mathcal P$ is given by the amalgamated free product 
\begin{equation}
\label{amalchir}
\Gamma(\mathcal{P}) 
\,=\,\Gamma(\mathcal{K})\ {\mathop{\ast}}_{\,\Gamma^{+}(\mathcal{F})} \,\Gamma(\mathcal{F})
\,=\,\Gamma(\mathcal{K})\, {\mathop{\ast}}_{\,\Gamma^{+}(\mathcal{F})} \,
(\Gamma^{+}(\mathcal{F})\rtimes C_2).
\end{equation}

We now have the following analogue of Theorem~\ref{freeextreg}.

\begin{theorem}
\label{freeextchir}
Let $\mathcal K$ be an oriented chiral $d$-polytope with directly regular facets, and suppose the automorphism group of $\mathcal K$ is a finite or an infinite 2-group. Then the free extension $\mathcal P$ of $\mathcal K$ is an oriented chiral $(d+1)$-polytope with an infinite 2-group as automorphism group.
\end{theorem}

\begin{proof}
We can argue as in the proof of Theorem~\ref{freeextreg}, now using the amalgamated free product in (\ref{amalchir}).  By the Torsion Theorem for amalgamated free products, each element of $\Gamma(\mathcal{P})$ of finite order is conjugate to an element of finite order in one of the two factors, $\Gamma(\mathcal{K})$ or $\Gamma(\mathcal{F})$. Since $\Gamma(\mathcal{K})$ is a finite or an infinite 2-group and the second factor is isomorphic to a subgroup of the first (and thus also is a finite or an infinite 2-group), it follows that every element of finite order in $\Gamma(\mathcal{P})$ is conjugate to an element of 2-power order in the factors $\Gamma(\mathcal{K})$ or $\Gamma(\mathcal{F})$, and thus itself must have 2-power order.
\end{proof}

Note that Theorem~\ref{freeextchir} is an extension theorem for {\it oriented\/} chiral polytopes. The free chiral extensions for the two possible oriented versions of the same underlying chiral polytope may not be isomorphic as polytopes in general. However, if the vertex-figures of $\mathcal K$ are directly regular, then the vertex-figures of the free extensions of Theorem~\ref{freeextchir} are also  directly regular (see~\cite[Thm.~2(e)]{SchWei1995}) and thus the free extensions for the two oriented versions of $\mathcal{K}$ will just coincide with the two enantiomorphic forms of the same underlying chiral $(d+1)$-polytope. For example, this happens if $d=3$, since the vertex-figures of chiral polyhedra are directly regular.

For a discussion of chiral extensions of chiral polytopes we refer to Cunningham \& Pellicer~\cite{CuPe2014} and Pellicer~\cite{Pel2025}.
\smallskip

There also is an analogue to Theorems~\ref{freeextreg} and~\ref{freeextchir} for alternating semiregular polytopes based on the existence of universal polytopes established in~\cite{MonSch2012,MonSch2021}). Suppose $\mathcal{P}$ and $\mathcal{Q}$ are regular $d$-polytopes, each with facets isomorphic to a regular $(d-1)$-polytope $\mathcal{K}$. Then there is a universal alternating semiregular $(d+1)$-polytope $\mathcal{U}$ with facets isomorphic to $\mathcal P$ and $\mathcal Q$. Here, $\mathcal{U}$ is regular if and only if $\mathcal{P}$ and $\mathcal{Q}$ are isomorphic. Let $\Gamma$ denote the free product of $\Gamma(\mathcal{P})$ and $\Gamma(\mathcal{Q})$, amalgamated along their common subgroups $\Gamma(\mathcal{K})$. Then the automorphism group of $\mathcal{U}$ is given by $\Gamma$ if $\mathcal{P}\not\cong\mathcal{Q}$ or $\Gamma\rtimes C_2$ if $\mathcal{P}\cong\mathcal{Q}$.

\begin{theorem}
\label{freeextsemireg}
Let $\mathcal{P}$ and $\mathcal{Q}$ be regular $d$-polytopes, each with facets isomorphic to a regular $(d-1)$-polytope $\mathcal{K}$. Suppose the automorphism groups of $\mathcal{P}$ and $\mathcal{Q}$ are finite or infinite 2-groups. Then the universal alternating semiregular $(d+1)$-polytope $\mathcal{U}$ with facets isomorphic to $\mathcal P$ and $\mathcal Q$ has an automorphism group which is an infinite 2-group.  
\end{theorem}

The proof works in much the same way as those of Theorems~\ref{freeextreg} and~\ref{freeextchir}, now applying the Torsion Theorem for amalgamated free products to $\Gamma$.

\section{Power polytopes}
\label{powerpols}

The power polytopes $2^{\mathcal K}$ and $2^{{\mathcal K},{\mathcal G}(2^m)}$ described in  \cite{McMSch2002} also have 2-groups as automorphism groups in certain cases.

If $\mathcal K$ is a finite regular $d$-polytope of type $\{2^{k_1},\ldots,2^{k_{d-1}}\}$ with $v$ vertices, then $2^{\mathcal K}$ is a regular $(d+1)$-polytope of type $\{4,2^{k_1},\ldots,2^{k_{d-1}}\}$ with vertex-figures isomorphic to $\mathcal{K}$ and automorphism group $\Gamma(2^{\mathcal K})=C_2^v\rtimes \Gamma({\mathcal K})$~(see \cite[8C2]{McMSch2002}). Thus, if $\Gamma(\mathcal{K})$ has order $2^n$, then $\Gamma(2^{\mathcal K})$ has order $2^{n+v}$.

If $m\geq 3$ and $\mathcal K$ is a centrally symmetric finite regular $d$-polytope of type $\{2^{k_1},\ldots,2^{k_{d-1}}\}$ with $v$ vertices, then $2^{{\mathcal K},{\mathcal G}(2^m)}$ is a regular $(d+1)$-polytope of type $\{4,2^{k_1},\ldots,2^{k_{d-1}}\}$ with vertex-figures isomorphic to $\mathcal{K}$ and automorphism group $\Gamma(2^{{\mathcal K},{\mathcal G}(2^m)})=(D_{2^m})^{v/2}\rtimes \Gamma({\mathcal K})$ (see \cite[8C5]{McMSch2002}). Recall here that a regular polytope is called {\it centrally symmetric\/} if its group contains a {\it proper\/} central involution $\alpha$, with ``proper" meaning that $\alpha$ does not fix a vertex (and thus pairs up the vertices in ``antipodal" pairs); that is, $\alpha$ must not lie in the vertex-figure subgroup $\langle\rho_1,\ldots,\rho_{d-1}\rangle$ of $\Gamma(\mathcal{K})=\langle\rho_0,\ldots,\rho_{d-1}\rangle$. Now, every 2-group has a nontrivial center. If $\Gamma(\mathcal{K})$ is of order $2^n$ and contains a proper central involution which also does not lie in the facet subgroup $\langle\rho_0,\ldots,\rho_{d-1}\rangle$ of $\Gamma(\mathcal{K})$ (so that $\mathcal{K}$ remains in tact as vertex-figure), then $\Gamma(2^{{\mathcal K},{\mathcal G}(2^m)})$ has a group of order $2^{n+(m+1)v/2}$.

\subsection*{Acknowledgment}
We are grateful for the opportunity to meet during Gabe Cunningham's and Egon Schulte visit at Beijing Jiaotong University and Shanxi Normal University, hosted by Yan-Quan Feng and Dong-Dong Hou in the summer of 2024. The present paper resulted from discussions during this visit. The work was supported by the National Natural Science Foundation of China (12201371, 12271024, 12331013, 12301461, 12425111).


\begin{thebibliography}{100}

\bibitem{BroLee2016} P.~A.~Brooksbank and D.~Leemans, Polytopes of large rank for $PSL(4,\mathbb{F}_q)$, \textit{J.\ Algebra} {\bf 452} (2016), 390--400.

\bibitem{BroVic2010} P.~A.~Brooksbank and D.~A.~Vicinsky, Three-dimensional classical groups acting on polytopes, \textit{Discrete Comp.\ Geom.} {\bf 44} (2010), no. 3, 654--659.

\bibitem{CFLM2017} P.~J.~Cameron, M.~E.~Fernandes, D.~Leemans and M.~Mixer, Highest rank of a polytope for $A_n$, \textit{Proc.\ Lond.\ Math.\ Soc.} (3) {\bf 115} (2017), no. 1, 135--176.

\bibitem{CFL2024} P.~J.~Cameron, M.~E.~Fernandes and D.~Leemans, The number of string C-groups of high rank, \textit{Adv.\ Math.} {\bf 453} (2024), Article ID 109832.

\bibitem{Con2013} M.~Conder, The smallest regular polytopes of given rank, \textit{Adv.\ Math.} {\bf 236} (2013), 92--110.

\bibitem{CHO2024} M.~Condera, I.~Hubard and E.~O’Reilly-Regueiro, Construction of chiral polytopes of large rank with alternating or symmetric automorphism group,  \textit{Adv.\ Math.} {\bf 452} (2024), Article ID 109819.

\bibitem{Cox1973} H.~S.~M.~Coxeter, \textit{Regular Polytopes}, 3rd Edition, Dover, New York, 1973.

\bibitem{Cox1940}  H.~S.~M.~Coxeter, Regular and semi-regular polytopes, I,  \textit{
Math.\ Z.}  {\bf 46} (1940), 380--407.  (In \textit{ Kaleidoscopes:  Selected
Writings of H.~S.~M.~Coxeter\/} (eds.\ F.~A.~Sherk, P.~McMullen, A.~C.~Thompson and 
A.~I.~Weiss), Wiley-Interscience (New York, etc., 1995), 251--278.)

\bibitem{CoxMos1980} H.~S.~M.~Coxeter and W.~O.~J.~Moser, \textit{Generators and Relations for Discrete Groups}, 4th Edition, Springer, 1980.

\bibitem{CuPe2014} G.~Cunningham and D.~Pellicer, Chiral extensions of chiral polytopes, \textit{Discrete Mathematics\/} {\bf 330} (2014), 51--60.

\bibitem{CFHS2025} G.~Cunningham, Y.-Q. Feng, D.-D. Hou and E.~Schulte, Small regular polytopes with 2-power order (title tentative), in preparation.

\bibitem{GLD2018} Y.~Gomi, M.L.~Loyola and M.~L.~A.~N. De Las Penas, String groups of order 1024, \textit{Contrib.\ Discrete Math.} {\bf 13} (2018), no. 1, 1--22.

\bibitem{HFL2019a}  D.-D. Hou, Y.-Q. Feng  and D.~Leemans, Existence of regular 3-polytopes of  order $2^n$, \textit{J. Group Theory} {\bf 22} (2019), 579--616.

\bibitem{HFL2019b}  D.-D. Hou, Y.-Q. Feng  and D.~Leemans, Existence of regular 3-hypertopes with $2^n$ chambers, \textit{Discrete Math.} {\bf 342} (2019), 1857--1863.

\bibitem{HFL2020}  D.-D. Hou, Y.-Q. Feng  and D.~Leemans, Regular polytopes of 2-power order, \textit{Discrete Comp.\ Geom.} {\bf 64} (2020), 339--346.

\bibitem{HFL2025}  D.-D. Hou, Y.-Q. Feng  and D.~Leemans, Regular 3-polytopes of order $2^{n}p$, \textit{J. Group Theory}, {\bf 28} (2025) 895-914.

\bibitem{HFLQ2025}  D.-D. Hou, Y.-Q. Feng, D.~Leemans and Hai-Peng Qu, String C-groups of order $4p^m$, \textit{Communications in Algebra}, {\bf 53} (2025), 3388-3399.

\bibitem{LynSch1977} R.~C.~Lyndon and P.~E.~Schupp, \textit{Combinatorial Group Theory\/}, Springer, Berlin, 1977.

\bibitem{LeeSch2007} D.~Leemans and E.~Schulte, Groups of type $L_2(q)$ acting on polytopes, \textit{Adv.\ Geom.} {\bf 7} (2007), no. 4, 529--539.

\bibitem{Loy2018} M.~L.~Loyola, String C-groups from groups of order $2^m$ and exponent $2^{m-3}$, preprint (2018), https://arxiv.org/abs/1607.01457.

\bibitem{McM2020} P.~McMullen, \textit{Geometric Regular Polytopes}, Cambridge University Press, 2020.

\bibitem{McMSch2002} P.~McMullen and E.~Schulte, \textit{Abstract Regular Polytopes}, Cambridge University Press, 2002.

\bibitem{MonSch2012} B.~Monson and E.~Schulte, Semiregular polytopes and amalgamated C-groups, \textit{Adv.\ Math.} {\bf 229} (2012), 2767--2791.

\bibitem{MonSch2020} B.~Monson and E.~Schulte, The assembly problem for alternating semiregular polytopes, \textit{Discrete Comput. Geom.} {\bf 64} (2020), 453--482.

\bibitem{MonSch2021} B.~Monson and E.~Schulte, Universal alternating semiregular polytopes, \textit{Can.\ J.\ Math.} {\bf 73} (2021), 572--595.

\bibitem{Pel2025} D.~Pellicer, \textit{Abstract Chiral Polytopes}, Cambridge University Press, 2025.

\bibitem{PisSer2013} T.~Pisanski and B.~Servatius, \textit{Configurations from a Graphical Viewpoint}, Birkh\"{a}user, Springer, 2013.

\bibitem{Sch1988} E.~Schulte, Amalgamation of regular incidence-polytopes, \textit{Proc. London Math. Soc.\/} (3) {\bf 56} (1988), 303--328.

\bibitem{SchWei1991} E.~Schulte and A.~I.~Weiss, Chiral polytopes, In \textit{Applied Geometry and Discrete Mathematics (The Victor Klee Festschrift)} (eds. P.~Gritzmann and B.~Sturmfels), DIMACS Series in Discrete Mathematics and Theoretical Computer Science, Vol. 4, American Mathematical Society and Association of Computing Machinery, 493--516, 1991.

\bibitem{SchWei1995} E.~Schulte and A.~I.~Weiss, Free extensions of chiral polytopes, \textit{Can. J. Math.} {\bf 47} (3) (1995), 641--654.

\bibitem{SchWei2006} E.~Schulte and A.~I.~Weiss, Problems on polytopes, their groups, and realizations, {\em Periodica Mathematica Hungarica\/}, (Special Issue on Discrete Geometry) {\bf 53}  (2006), 231--255.

\end{thebibliography}
\end{document}